\documentclass[12pt,reqno]{amsart}

\usepackage[english]{babel}
\usepackage[T1]{fontenc}
\usepackage[utf8]{inputenc}
\usepackage{amsmath,amsfonts,amssymb,amsthm}
\usepackage{cite}
\usepackage{url}
\usepackage{lmodern}
\usepackage{enumitem}
\usepackage[hidelinks]{hyperref}
\usepackage{geometry}

\numberwithin{equation}{section}

\DeclareMathOperator{\Der}{Der}
\DeclareMathOperator{\LocDer}{LocDer}

\newtheorem{theorem}{Theorem}[section]
\newtheorem{proposition}[theorem]{Proposition}
\newtheorem{lemma}[theorem]{Lemma}
\newtheorem{corollary}[theorem]{Corollary}

\theoremstyle{definition}
\newtheorem{definition}[theorem]{Definition}
\newtheorem{remark}[theorem]{Remark}

\newcommand{\C}{\mathbb C}

\newcommand{\Id}{\operatorname{id}}

\setlist[enumerate]{label=(\arabic*),topsep=3pt,itemsep=0pt}
\allowdisplaybreaks[1]
\DeclareMathOperator{\Aut}{Aut}
\DeclareMathOperator{\LocAut}{LocAut}
\DeclareMathOperator{\Inn}{Inn}

\title[Local maps on nilpotent Lie and Leibniz algebras]{Local Derivations and Local Automorphisms of Arbitrary
Finite-Dimensional Nilpotent Lie and Leibniz Algebras}

\author[Kh. Atajonov]{Khusainboy Atajonov$^1$}
\author[A. Khudoyberdiyev]{Abror Khudoyberdiyev$^{1,2}$}
\author[B. Yusupov]{Bakhtiyor Yusupov$^{2,3}$}

\address{$^1$ National University of Uzbekistan, 4 University Street,
Tashkent 100174, Uzbekistan}

\address{$^2$ V.I. Romanovskiy Institute of Mathematics, Academy of Sciences
of Uzbekistan, 9 University Street, Tashkent 100174, Uzbekistan}

\address{$^3$ Department of Applied Mathematics and Mathematical Physics, Urgench State University, Urgench, Uzbekistan}

\email{atajonovxusainboy@gmail.com}
\email{khabror@mail.ru}
\email{baxtiyor\_yusupov\_93@mail.ru}

\begin{document}
\begin{abstract}
We study local derivations and local automorphisms of
finite-dimensional nilpotent complex Lie and right Leibniz
algebras. We prove that every non-abelian Lie algebra in this
class admits both a local derivation which is not a derivation
and a local automorphism which is not an automorphism.
The constructions use a derivation $D$ satisfying
$0\ne D([L,L])\subseteq Z(L)$.
For non-abelian right Leibniz algebras, both existence statements
hold whenever the dimension is at most seven or the nilpotency
class is at most four. We construct an eight-dimensional
non-Lie algebra of class five satisfying
$\LocDer(A)=\Der(A)$ and $\LocAut(A)=\Aut(A)$.
Consequently, eight and five are the least possible dimension
and nilpotency class, respectively, for either equality.
\end{abstract}
\maketitle


\section{Introduction}\label{sec:introduction}

Local derivations and local automorphisms describe maps whose values
at individual elements agree with those of derivations and automorphisms,
respectively. More precisely, a linear map $\Delta:L\to L$ is a local
derivation if, for every $x\in L$, there exists a derivation $D_x$ such
that $\Delta(x)=D_x(x)$. Local automorphisms are defined similarly,
with automorphisms in place of derivations. The witnessing maps may
depend on $x$, and the local map need not itself be a derivation or
an automorphism. We call such a local map proper. These notions
originate in the work of Kadison~\cite{kadison1990} and Larson and
Sourour~\cite{larson1990} on operator algebras.

For Lie algebras, the relation between local and global maps depends
on the algebraic structure. Zhao and Wang~\cite{zhao2015triangular}
studied local derivations and local automorphisms of upper triangular
matrix Lie algebras over commutative rings. Wang and
Wang~\cite{wang2019antisymmetric} considered local derivations of
Lie algebras of antisymmetric matrices. Over algebraically closed
fields of characteristic zero, Ayupov and
Kudaybergenov~\cite{ayupov2016} proved that every local derivation
of a finite-dimensional semisimple Lie algebra is a derivation,
and exhibited nilpotent Lie algebras admitting proper local
derivations. Further rigidity results were obtained for standard
Borel subalgebras~\cite{yu2020} and solvable Lie algebras of maximal
rank~\cite{kudaybergenov2022}. For general solvable Lie algebras,
both rigidity and the existence of proper local derivations
occur~\cite{ayupov2021}. Related results for Witt algebras, Witt
algebras in prime characteristic, and conformal Galilei algebras
were established in~\cite{chen2022witt,yao2022witt,alauadinov2024galilei}.

The automorphism problem requires a separate analysis. Following
the investigation of local automorphisms on Lie and Leibniz
algebras in~\cite{ayupov2018}, Costantini~\cite{costantini2019}
proved that every local automorphism of a finite-dimensional
simple Lie algebra over an algebraically closed field of
characteristic zero is either an automorphism or an
anti-automorphism. By contrast, every local automorphism of
a complex solvable Lie algebra of maximal rank is an
automorphism~\cite{kudaybergenov2024}. Chen, Zhao and
Zhao~\cite{chen2021wittaut} obtained rigidity results for invertible
local automorphisms of simple generalized Witt algebras.

Local derivations have also been studied with values in modules.
Wang, Tang and Li~\cite{wang2025sl2modules} described local
derivations of $\mathfrak{sl}(2)$ with values in finite-dimensional
completely reducible modules. Zhou, Wang and
Liu~\cite{zhou2025modularsl2} considered local derivations of the
modular Lie algebra $\mathfrak{sl}(2)$ with values in simple
modules. These results exhibit the role of the target module
in the distinction between local derivations and derivations.

The corresponding questions extend to Lie superalgebras.
Local superderivations of basic classical Lie superalgebras and
of the strange Lie superalgebra $\mathfrak{q}(n)$ were studied
in~\cite{chen2017classicalsuper,chen2018qn}, while the Cartan type
case was treated in~\cite{yuan2021cartan}. Wu, Gao, Liu and
Ye~\cite{wu2024supervirasoro} investigated local derivations and
local automorphisms of super Virasoro algebras. Sheng, Liu and
Liu~\cite{sheng2024filiformsuper} described local automorphisms
and local superderivations of model filiform Lie superalgebras.
Beyond the Lie setting, local derivations and local automorphisms
of Cayley algebras were described
in~\cite{ayupov2023cayley}, and local derivations of filiform
associative algebras were studied
in~\cite{abdurasulov2025associative}.

The immediate starting point of the present paper is the work
of Khudoyberdiyev and Jumaniyozov~\cite{khudoyberdiyev2025}
on nilpotent Lie algebras. They established the existence of
proper local derivations and proper local automorphisms for
nilpotency indices three and four, and obtained a sufficient
condition for both existence statements for algebras of higher
nilpotency index. Their applications include two-generated nilpotent
Lie algebras. With the convention used here, their nilpotency
indices three and four correspond to nilpotency classes two
and three. This raises the question of whether the additional
hypotheses can be removed.

We prove that every finite-dimensional non-abelian nilpotent
complex Lie algebra admits both a proper local derivation and
a proper local automorphism. Our results therefore include the
complex existence results of~\cite{khudoyberdiyev2025}, without
restrictions on the nilpotency class or the number of generators.
Consequently, among finite-dimensional nilpotent complex Lie
algebras, each of the equalities
$\LocDer(L)=\Der(L)$ and $\LocAut(L)=\Aut(L)$ is equivalent
to abelianity.

The main structural step is the construction of a derivation
$D\in\Der(L)$ such that
\[
0\ne D(L^2)\subseteq Z(L).
\]
Here $L^2=[L,L]$. Composing $D$ with a projection onto $L^2$
produces a proper local derivation with central image.
For class two, the auxiliary derivation is constructed directly.
In higher nilpotency classes, we use the centraliser of $L^2$
and compositions of inner derivations. A composition property
preserves the derivation identity, and maximality of a composition
that remains nonzero on $L^2$ forces its image there to be central.
The same auxiliary derivation also yields proper local
automorphisms through a separate construction of invertible
bracket-preserving witnesses.

Leibniz algebras provide a natural setting in which to examine
the scope of these conclusions. They generalise Lie algebras by
allowing a non-skew-symmetric bracket. In the right Leibniz
convention used in this paper, every right multiplication is
a derivation. Local maps have been investigated for several
classes of Leibniz algebras. Ayupov, Kudaybergenov and
Omirov~\cite{ayupov2020simple} proved rigidity of local derivations
on complex simple Leibniz algebras and established the existence
of proper local derivations on filiform Leibniz algebras of
dimension at least three. The semisimple case was treated by
Kaygorodov, Kudaybergenov and Yuldashev~\cite{kaygorodov2022leibniz}.
For solvable Leibniz algebras, local derivations were studied
in~\cite{ayupov2020solvable}, and local automorphisms for several
nilradicals were investigated
in~\cite{arzikulov2022leibnizaut,karimjanov2024leibnizaut}.

In the nilpotent Leibniz setting, local derivations of
$p$-filiform algebras were described
in~\cite{ayupov2020pfiliform}, and local automorphisms of some
$p$-filiform families were studied
in~\cite{yusupov2025pfiliformaut}. Local derivations and local
automorphisms of direct sums of null-filiform Leibniz algebras
were investigated in~\cite{adashev2022}. These results for
particular nilpotent families, together with our general Lie
theorems, lead to the following question: does every
finite-dimensional non-abelian nilpotent right Leibniz algebra
admit proper local derivations and proper local automorphisms?
The absence of skew-symmetry prevents a direct application
of the Lie arguments; in particular, $[A^i,A^j]$ and
$[A^j,A^i]$ need not coincide.

We show that both existence statements hold for every
finite-dimensional non-abelian nilpotent complex right
Leibniz algebra whose dimension is at most seven or whose
nilpotency class is at most four. They fail in general:
we exhibit an eight-dimensional non-Lie right Leibniz
algebra $A_8$ of class five satisfying
\[
\LocDer(A_8)=\Der(A_8),\qquad
\LocAut(A_8)=\Aut(A_8).
\]
Thus eight is the least possible dimension and five is the
least possible nilpotency class for either equality among
non-abelian nilpotent complex right Leibniz algebras.
The same algebra attains both bounds. The two rigidity
equalities are proved separately, using explicit descriptions
of the derivations and automorphisms of $A_8$.

The paper is organised as follows. Section~2 introduces
the notation and preliminary facts. Sections~3 and~4 prove
the existence theorems for local derivations and local
automorphisms of nilpotent Lie algebras. Section~5 treats
local derivations of nilpotent right Leibniz algebras and
establishes the sharp dimension and nilpotency-class bounds.
Section~6 proves the corresponding results for local
automorphisms.

\section{Preliminaries}\label{sec:preliminaries}

Throughout, all algebras are finite-dimensional over $\C$.
A right Leibniz algebra is a vector space $L$ equipped with
a bilinear bracket satisfying
\[
[x,[y,z]]=[[x,y],z]-[[x,z],y]
\qquad(x,y,z\in L).
\]
A right Leibniz algebra is a Lie algebra if $[x,x]=0$
for every $x\in L$. In this case, the bracket is
skew-symmetric and satisfies the Jacobi identity.

For subspaces $U,V\subseteq L$, the notation $[U,V]$
means the linear span of their brackets. We use the
lower central series
\[
L^1=L,\qquad L^{j+1}=[L^j,L]\quad(j\ge1).
\]
In particular, $L^2=[L,L]$ is the derived algebra.
For Lie algebras, this agrees with the convention
$L^{j+1}=[L,L^j]$.

The algebra is nilpotent of class $s$ if $L^s\ne0$
and $L^{s+1}=0$. In this convention, the nilindex used
in~\cite{khudoyberdiyev2025} is $s+1$. In particular,
their nilindices three and four correspond to classes
two and three here. A nonzero abelian algebra has
class one. We do not assign a nilpotency class to
the zero algebra.

The two-sided centre of $L$ is
\[
Z(L)=\{z\in L:[z,L]=[L,z]=0\}.
\]
For a nilpotent algebra of class $s$, we have
$L^s\subseteq Z(L)$. Each term of the lower central
series is an ideal, and $L^{j+1}\subseteq L^j$.

When $L$ is a Lie algebra, the centralizer of
a subspace $U\subseteq L$ is
\[
C_L(U)=\{z\in L:[z,U]=0\}.
\]

\begin{definition}
A linear map $D:L\to L$ is a derivation if
\[
D([a,b])=[D(a),b]+[a,D(b)]
\qquad(a,b\in L).
\]
The space of derivations is denoted by $\Der(L)$.
For a right Leibniz algebra $L$, every right multiplication
$R_c:x\mapsto[x,c]$ is a derivation. These are called
inner derivations, and their space is denoted by
$\Inn(L)=\{R_c:c\in L\}$.
When $L$ is a Lie algebra, we also use the notation
$(\operatorname{ad}c)(x)=[c,x]$; thus
$\operatorname{ad}c=-R_c$.
A linear map $\Delta:L\to L$ is a local derivation if,
for every $x\in L$, there is $D_x\in\Der(L)$ such that
$\Delta(x)=D_x(x)$. We denote this space by $\LocDer(L)$.
A local derivation is proper if it is not a derivation.
\end{definition}

\begin{definition}
An automorphism of $L$ is a bijective linear map $A:L\to L$
satisfying
\[
A([a,b])=[A(a),A(b)]
\qquad(a,b\in L).
\]
Its group is denoted by $\Aut(L)$.
A linear map $\Phi:L\to L$ is a local automorphism if,
for every $x\in L$, there is $A_x\in\Aut(L)$ such that
$\Phi(x)=A_x(x)$. We denote their set by $\LocAut(L)$
and call $\Phi$ proper when it is not an automorphism.
\end{definition}

Every local automorphism in finite dimension is bijective:
if $\Phi(x)=0$, a witnessing automorphism gives $A_x(x)=0$,
hence $x=0$. Injectivity then implies surjectivity.
We always have
\[
\Der(L)\subseteq\LocDer(L),\qquad
\Aut(L)\subseteq\LocAut(L).
\]

\begin{lemma}\label{lem:central-maps}
Let $T:L\to Z(L)$ be a linear map. Then:
\begin{itemize}
\item $T$ is a derivation if and only if $T(L^2)=0$.
\item If $f\in L^*$ vanishes on $L^2$ and $z\in Z(L)$,
then the map $T$ defined by $T(x)=f(x)z$ is a derivation.
If, in addition, $f(z)=0$, then $T^2=0$ and $\Id+T$
is an automorphism with inverse $\Id-T$.
\end{itemize}
\end{lemma}

\begin{proof}
Since $T(L)\subseteq Z(L)$, for all $a,b\in L$ we have
\[
[T(a),b]+[a,T(b)]=0.
\]
Thus $T$ is a derivation if and only if
\[
T([a,b])=0\qquad\text{for all }a,b\in L.
\]
Since $L^2=[L,L]$ is spanned by brackets and $T$ is linear,
this is equivalent to $T(L^2)=0$.

Now let $f\in L^*$ satisfy $f(L^2)=0$, let $z\in Z(L)$,
and define $T(x)=f(x)z$. Then $T(L)\subseteq Z(L)$ and
$T(L^2)=0$, so the first assertion shows that $T$ is
a derivation.

If, in addition, $f(z)=0$, then for every $x\in L$,
\[
T^2(x)=T(f(x)z)=f(x)f(z)z=0.
\]
Hence $T^2=0$, and
\[
(\Id+T)(\Id-T)=(\Id-T)(\Id+T)=\Id.
\]
Therefore $\Id+T$ is bijective with inverse $\Id-T$.
Moreover, centrality of $T(L)$ and the equality $T(L^2)=0$
give
\[
[(\Id+T)(a),(\Id+T)(b)]
=[a+T(a),b+T(b)]
=[a,b]
=(\Id+T)([a,b]).
\]
Thus $\Id+T$ preserves the bracket and is an automorphism.
\end{proof}

The constructions of local derivations and local automorphisms
below use a projection onto $L^2$ and, for each $x\notin L^2$,
a linear functional vanishing on $L^2$ and taking the value
$1$ at $x$. We record these elementary facts in the following
remark.

\begin{remark}
For each vector-space complement $V$ of $L^2$, write
$L=V\oplus L^2$ and let $\pi:L\to L^2$ be the corresponding
projection, also viewed as an endomorphism of $L$.
Then $\pi^2=\pi$ and $\pi|_{L^2}=\Id_{L^2}$.
If $x\notin L^2$, a functional on $L/L^2$ supplies
$f\in L^*$ with $f(L^2)=0$ and $f(x)=1$.
\end{remark}

\section{Local derivations}\label{subsec:lie-local-derivations}

In this section, we prove that every finite-dimensional non-abelian
nilpotent complex Lie algebra admits a local derivation which is not
a derivation. Our aim is to construct such a map with central image
and nonzero restriction to the derived algebra. By
Lemma~\ref{lem:central-maps}, these two properties ensure that the
map is not a derivation; the remaining task is to verify locality.

The construction starts with a derivation
$D\in\Der(L)$ satisfying
\begin{equation}\label{eq:nilpotent-central-criterion}
0\ne D(L^2)\subseteq Z(L).
\end{equation}
Composing $D$ with a projection onto $L^2$ produces a linear map
with the required properties. We will show that this map is a local
derivation by constructing a suitable derivation at each element
of $L$. Thus, the essential step is to establish the existence
of $D$; merely requiring $D\ne0$ is insufficient, since a nonzero
derivation may vanish on $L^2$.

We prove this existence separately for nilpotency classes $s=2$
and $s\ge3$. The first case admits a direct construction, while
the second uses the centralizer of $L^2$ and compositions of inner
derivations.

\begin{lemma}\label{lem:class-two-central-derivation}
Every finite-dimensional complex nilpotent Lie algebra of class two
admits a derivation satisfying \eqref{eq:nilpotent-central-criterion}.
\end{lemma}

\begin{proof}
Since $L$ has nilpotency class $s=2$, we have $L^2\ne0$
and $[L,L^2]=L^3=0$. Hence $L^2\subseteq Z(L)$. Choose a complement $V$ of $L^2$
and let $\pi:L\to L^2$ be the corresponding projection.
Set $D=\Id+\pi$. Since $\pi|_{L^2}=\Id_{L^2}$ and $\pi(L)$ is central,
\[
D([a,b])=2[a,b]
=[a+\pi(a),b]+[a,b+\pi(b)]
=[D(a),b]+[a,D(b)]
\]
for all $a,b\in L$. Thus $D$ is a derivation and $D(L^2)=L^2$. Since $L^2\ne0$ and
$L^2\subseteq Z(L)$, $D$ satisfies \eqref{eq:nilpotent-central-criterion}.
\end{proof}

For $s\ge3$, we use the following composition property.

\begin{lemma}\label{lem:centraliser-composition}
Let $L$ be a Lie algebra and set
$C=C_L(L^2)=\{u\in L:[u,L^2]=0\}.$
Then $C$ is an ideal of $L$. Moreover, if $D\in\Der(L)$
satisfies $D(L)\subseteq C$, then $\operatorname{ad}c\circ D$
is a derivation with image in $C$ for every $c\in L$.
\end{lemma}

\begin{proof}
For $x\in L$, $y\in C$, and $w\in L^2$, the Jacobi identity gives
\[
[[x,y],w]=[x,[y,w]]-[y,[x,w]]=0,
\]
since $[y,w]=0$ and $[x,w]\in L^2$. Thus $[x,y]\in C$,
so $C$ is an ideal.

Fix $c\in L$ and set $F=\operatorname{ad}c\circ D$.
For all $a,b\in L$, we have
\begin{align*}
F([a,b])
&=[c,[D(a),b]]+[c,[a,D(b)]]\\
&=[[c,D(a)],b]+[D(a),[c,b]]+[[c,a],D(b)]+[a,[c,D(b)]]\\
&=[F(a),b]+[a,F(b)].
\end{align*}
The middle terms vanish because $D(a),D(b)\in C$ and
$[c,a],[c,b]\in L^2$. Hence $F$ is a derivation, and
\[
F(L)=[c,D(L)]\subseteq[L,C]\subseteq C.
\]
\end{proof}

\begin{lemma}\label{lem:class-three-plus-central-derivation}
Every finite-dimensional complex nilpotent Lie algebra of class
$s\ge3$ admits a derivation satisfying
\eqref{eq:nilpotent-central-criterion}. This derivation can be chosen
as a composition of inner derivations with image in $C_L(L^2)$.
\end{lemma}

\begin{proof}
Set $C=C_L(L^2)$. Since $s\ge3$, we have $[L,L^2]=L^3\ne0$,
so $C\ne L$. On the other hand, $L^s\subseteq Z(L)\subseteq C$.
Thus there is a largest integer $r\ge1$ such that
$L^r\not\subseteq C$, and its maximality gives
$L^{r+1}\subseteq C$.

Choose $y\in L^r\setminus C$ and set $D_0=\operatorname{ad}y$.
Then $D_0$ is an inner derivation and
\[
D_0(L)=[y,L]\subseteq L^{r+1}\subseteq C.
\]
Since $y\notin C$, there exists $w\in L^2$ with $[y,w]\ne0$.
Hence $D_0(L^2)\ne0$. Since $D_0\in\Der(L)$ and $D_0(L)\subseteq C$,
$D_0$ satisfies the hypotheses of
Lemma~\ref{lem:centraliser-composition}.

For $k\ge0$ and $c_1,\ldots,c_k\in L$, consider
\[
D_k=\operatorname{ad}c_k\circ\cdots\circ
\operatorname{ad}c_1\circ D_0,
\]
where $k=0$ means $D_k=D_0$. By repeated application of
Lemma~\ref{lem:centraliser-composition}, every such map is
a derivation with image in $C$. Each inner derivation maps
$L^j$ into $L^{j+1}$; starting from
$D_0(L)\subseteq L^{r+1}$, we therefore obtain
\[
D_k(L)\subseteq L^{r+k+1}.
\]
In particular, $D_k=0$ whenever $k\ge s-r$, since $L^{s+1}=0$.

Consequently, there is a largest integer $k\ge0$ for which
some choice of $c_1,\ldots,c_k$ yields $D_k(L^2)\ne0$.
Such an integer exists because $D_0(L^2)\ne0$.
Fix a corresponding choice and put $D=D_k$.
Then $D\in\Der(L)$, $D(L)\subseteq C$, and $D(L^2)\ne0$.

It remains to show that $D(L^2)\subseteq Z(L)$.
Suppose that some $w\in L^2$ satisfies $D(w)\notin Z(L)$.
Then there exists $c\in L$ such that $[c,D(w)]\ne0$.

Set $\widetilde D=\operatorname{ad}c\circ D$.
Since $D\in\Der(L)$ and $D(L)\subseteq C$,
Lemma~\ref{lem:centraliser-composition} implies that
$\widetilde D$ is a derivation with image in $C$. Moreover,
\[
\widetilde D(w)=[c,D(w)]\ne0,
\]
so $\widetilde D(L^2)\ne0$. However,
\[
\widetilde D
=\operatorname{ad}c\circ\operatorname{ad}c_k\circ\cdots\circ
\operatorname{ad}c_1\circ D_0.
\]
Taking $c_{k+1}=c$, we obtain a composition with $k+1$
inner derivations after $D_0$ whose restriction to $L^2$ is nonzero.
This contradicts the choice of $k$ as the largest integer
for which such a composition exists.

Therefore $D(L^2)\subseteq Z(L)$. Together with $D(L^2)\ne0$,
this shows that $D$ satisfies
\eqref{eq:nilpotent-central-criterion}.
\end{proof}

We next use this derivation to construct a proper local derivation
and establish the main theorem of this section.

\begin{lemma}\label{lem:central-criterion-local}
Let $L$ be a finite-dimensional complex Lie algebra admitting
a derivation satisfying \eqref{eq:nilpotent-central-criterion}.
Then $L$ admits a local derivation $\Delta$ such that
\[
\Delta(L)\subseteq Z(L),\qquad \Delta(L^2)\ne0.
\]
In particular, $\Delta$ is not a derivation.
\end{lemma}

\begin{proof}
Choose a vector-space complement $V$ of $L^2$ in $L$, and let
$\pi:L\to L^2$ be the projection associated with $L=V\oplus L^2$.
Define $\Delta=D\circ\pi.$
Since $D$ and $\pi$ are linear, $\Delta$ is linear. Moreover,
$\pi(L)\subseteq L^2$ and $D(L^2)\subseteq Z(L)$ imply that
$\Delta(L)\subseteq Z(L)$. Since $\pi(w)=w$ for every $w\in L^2$,
we also have $\Delta|_{L^2}=D|_{L^2}$.

To verify locality, fix $x\in L$. If $x\in L^2$, then
$\Delta(x)=D(x)$, so $D$ is a suitable derivation at $x$.
If $x\notin L^2$, choose $f\in L^*$ satisfying
$f(L^2)=0$ and $f(x)=1$, and define
\[
D_x(z)=f(z)\Delta(x)\qquad(z\in L).
\]
The map $D_x$ has central image and vanishes on $L^2$.
Hence it is a derivation by Lemma~\ref{lem:central-maps}, and
\[
D_x(x)=f(x)\Delta(x)=\Delta(x).
\]
Thus, for every $x\in L$, there exists a derivation agreeing
with $\Delta$ at $x$, proving that $\Delta$ is a local derivation.

Finally, $\Delta|_{L^2}=D|_{L^2}$ gives $\Delta(L^2)=D(L^2)\ne0$.
Since $\Delta(L)\subseteq Z(L)$, Lemma~\ref{lem:central-maps}
shows that $\Delta$ is not a derivation.
\end{proof}

\begin{theorem}\label{thm:nilpotent-proper-local}
Every finite-dimensional non-abelian nilpotent complex Lie algebra
$L$ admits a local derivation $\Delta$ satisfying
\[
\Delta(L)\subseteq Z(L),\qquad \Delta([L,L])\ne0.
\]
In particular, $\Delta$ is not a derivation.
\end{theorem}

\begin{proof}
Since $L$ is non-abelian, its nilpotency class is at least two.
Lemmas~\ref{lem:class-two-central-derivation} and
\ref{lem:class-three-plus-central-derivation} therefore provide
a derivation satisfying \eqref{eq:nilpotent-central-criterion}.
The conclusion follows from Lemma~\ref{lem:central-criterion-local}.
\end{proof}

\section{Local automorphisms}\label{subsec:lie-local-automorphisms}

In this section, we prove that every finite-dimensional non-abelian
nilpotent complex Lie algebra admits a local automorphism which
is not an automorphism. As in the preceding section, we construct
a linear map differing from the identity by a central-image map
which acts nontrivially on the derived algebra. We then verify
bijectivity and construct an automorphism agreeing with this map
at each element of $L$.

We treat nilpotency classes $s=2$ and $s\ge3$
separately. For $s\ge3$, we use the derivation constructed in
Lemma~\ref{lem:class-three-plus-central-derivation}, whose form
as a composition of inner derivations provides the additional
properties needed below.

\begin{lemma}\label{lem:class-two-local-aut}
Let $L$ be a finite-dimensional complex nilpotent Lie algebra
of class two. Then $L$ admits a proper local automorphism $\Phi$
such that
\[
(\Phi-\Id)(L)\subseteq Z(L).
\]
\end{lemma}

\begin{proof}
Since $L$ has nilpotency class two, $L^2\ne0$ and
$[L,L^2]=L^3=0$, so $L^2\subseteq Z(L)$.
Choose a vector-space complement $V$ of $L^2$ in $L$,
and let $\pi:L\to L^2$ be the corresponding projection. Set
\[
\Phi=\Id+\pi.
\]
Since $\pi^2=\pi$, we have
\[
(\Id+\pi)\bigl(\Id-\tfrac12\pi\bigr)
=\bigl(\Id-\tfrac12\pi\bigr)(\Id+\pi)=\Id.
\]
Thus $\Phi$ is bijective with inverse $\Id-\tfrac12\pi$.
Moreover, $(\Phi-\Id)(L)=\pi(L)=L^2\subseteq Z(L)$.

To verify locality, fix $x\in L$. If $x\in L^2$, take
\[
A_x=\sqrt2\,\Id+(2-\sqrt2)\pi.
\]
For $v\in V$ and $w\in L^2$, we have $\pi(v)=0$ and $\pi(w)=w$,
so
\[
A_x(v+w)=\sqrt2(v+w)+(2-\sqrt2)w=\sqrt2v+2w.
\]
Thus $A_x$ is bijective, with inverse given by
\[
A_x^{-1}(v+w)=\frac1{\sqrt2}v+\frac12w.
\]
Since $\pi(L)\subseteq Z(L)$ and
$\pi|_{L^2}=\Id_{L^2}$, for all $a,b\in L$ we obtain
\[
[A_x(a),A_x(b)]=2[a,b]=A_x([a,b]).
\]
Hence $A_x$ is an automorphism, and
$A_x(x)=2x=\Phi(x)$.

If $x\notin L^2$, choose $f\in L^*$ with $f(L^2)=0$ and $f(x)=1$,
and define
\[
E_x(z)=f(z)\pi(x)\qquad(z\in L).
\]
Since $\pi(x)\in L^2 \subseteq Z(L)$ and $f(\pi(x))=0$,
Lemma~\ref{lem:central-maps} shows that
$A_x=\Id+E_x$ is an automorphism with inverse $\Id-E_x$.
Moreover,
\[
A_x(x)=x+f(x)\pi(x)=x+\pi(x)=\Phi(x).
\]
Thus $\Phi$ is a local automorphism.

Finally, choose $a,b\in L$ with $[a,b]\ne0$.
Since $\pi(L)$ is central and $\pi$ fixes $L^2$,
\[
[\Phi(a),\Phi(b)]=[a,b],
\qquad
\Phi([a,b])=2[a,b]\ne[a,b].
\]
Therefore $\Phi$ is not an automorphism.
\end{proof}

For the case $s\ge3$, we first record a criterion for obtaining
an automorphism from a nilpotent derivation.

\begin{lemma}\label{lem:nilpotent-derivation-automorphism}
Let $L$ be a Lie algebra and let $D\in\Der(L)$ satisfy
\[
[D(L),D(L)]=0,\qquad D^q=0
\]
for some integer $q\ge1$. Then $\Id+D$ is an automorphism, with
\begin{equation}\label{eq:nilpotent-geometric-inverse}
(\Id+D)^{-1}=\sum_{i=0}^{q-1}(-D)^i.
\end{equation}
\end{lemma}

\begin{proof}
For all $a,b\in L$, the derivation identity gives
\begin{align*}
[(\Id+D)(a),(\Id+D)(b)]
&=[a,b]+[D(a),b]+[a,D(b)]+[D(a),D(b)]\\
&=[a,b]+D([a,b])\\
&=(\Id+D)([a,b]).
\end{align*}
Thus $\Id+D$ preserves the bracket. Moreover,
\[
(\Id+D)\sum_{i=0}^{q-1}(-D)^i
=\left(\sum_{i=0}^{q-1}(-D)^i\right)(\Id+D)
=\Id-(-D)^q=\Id.
\]
Hence $\Id+D$ is bijective with inverse
\eqref{eq:nilpotent-geometric-inverse}.
\end{proof}

\begin{lemma}\label{lem:class-three-plus-local-aut}
Let $L$ be a finite-dimensional complex nilpotent Lie algebra
of class $s\ge3$. Then $L$ admits a proper local automorphism
$\Phi$ such that
\[
(\Phi-\Id)(L)\subseteq Z(L),\qquad (\Phi-\Id)^2=0.
\]
\end{lemma}

\begin{proof}
Let $D$ be the derivation constructed in
Lemma~\ref{lem:class-three-plus-central-derivation}.
It satisfies \eqref{eq:nilpotent-central-criterion},
has image in $C_L(L^2)$, and is a composition of inner derivations.

Every inner derivation maps $L$ into $L^2$, annihilates $Z(L)$,
and maps $L^j$ into $L^{j+1}$. Consequently,
\[
D(L)\subseteq L^2\cap C_L(L^2),\qquad D(Z(L))=0,
\qquad D(L^j)\subseteq L^{j+1}.
\]
The last inclusion implies $D^s=0$. Also,
$[D(L),D(L)]=0$, since $D(L)\subseteq L^2$ and every element
of $D(L)$ centralizes $L^2$. By
Lemma~\ref{lem:nilpotent-derivation-automorphism},
$\Id+D$ is therefore an automorphism.

Choose $L=V\oplus L^2$, let $\pi:L\to L^2$ be the corresponding
projection, and set
\[
\Delta=D\circ\pi,\qquad \Phi=\Id+\Delta.
\]
Since $D(L^2)\subseteq L^2\cap Z(L)$, we have
$\Delta(L)\subseteq L^2\cap Z(L)$. As $\pi$ fixes $L^2$ and
$D$ annihilates $Z(L)$,
\[
\Delta^2=D^2\circ\pi=0.
\]
Hence $\Phi$ is bijective with inverse $\Id-\Delta$, and
the two asserted properties of $\Phi-\Id$ hold.

To verify locality, fix $x\in L$. If $x\in L^2$, take
$A_x=\Id+D$. Since $\pi(x)=x$,
\[
A_x(x)=x+D(x)=x+\Delta(x)=\Phi(x).
\]
If $x\notin L^2$, choose $f\in L^*$ with $f(L^2)=0$ and $f(x)=1$,
and define
\[
E_x(z)=f(z)\Delta(x)\qquad(z\in L).
\]
Here $\Delta(x)\in L^2\cap Z(L)$, so $f(\Delta(x))=0$.
Lemma~\ref{lem:central-maps} implies that
$A_x=\Id+E_x$ is an automorphism with inverse $\Id-E_x$.
Moreover,
\[
A_x(x)=x+f(x)\Delta(x)=\Phi(x).
\]
Thus $\Phi$ is a local automorphism.

Finally, since $D(L^2)\ne0$ and $L^2$ is spanned by brackets,
there exist $a,b\in L$ with $D([a,b])\ne0$.
Centrality of $\Delta(L)$ and the equality $\Delta|_{L^2}=D|_{L^2}$ give
\[
[\Phi(a),\Phi(b)]=[a,b],
\qquad
\Phi([a,b])=[a,b]+D([a,b])\ne[a,b].
\]
Therefore $\Phi$ is not an automorphism.
\end{proof}

The preceding lemmas yield the main theorem of this section.

\begin{theorem}\label{thm:nilpotent-proper-local-aut}
Every finite-dimensional non-abelian nilpotent complex Lie algebra
$L$ admits a proper local automorphism $\Phi$ satisfying
\[
(\Phi-\Id)(L)\subseteq Z(L).
\]
If the nilpotency class is at least three, $\Phi$ can additionally
be chosen so that $(\Phi-\Id)^2=0$. In particular,
\[
\operatorname{Aut}(L)\subsetneq\operatorname{LocAut}(L).
\]
\end{theorem}

\begin{proof}
The existence and the asserted properties follow from
Lemmas~\ref{lem:class-two-local-aut} and
\ref{lem:class-three-plus-local-aut}.
Since every automorphism is a local automorphism, the existence
of a proper local automorphism gives the strict inclusion.
\end{proof}

\begin{corollary}\label{cor:nilpotent-local-rigidity}
Let $L$ be a finite-dimensional nilpotent complex Lie algebra.
The following conditions are equivalent:
\begin{enumerate}
\item $L$ is abelian;
\item every local derivation of $L$ is a derivation;
\item every local automorphism of $L$ is an automorphism.
\end{enumerate}
\end{corollary}

\begin{proof}
If $L$ is abelian, every linear map $L\to L$ is a derivation,
and every bijective linear map $L\to L$ is an automorphism.
Since every local automorphism is bijective in finite dimension,
the first condition implies the other two.

If $L$ is non-abelian,
Theorems~\ref{thm:nilpotent-proper-local} and
\ref{thm:nilpotent-proper-local-aut} provide a local derivation
which is not a derivation and a local automorphism which is not
an automorphism, respectively. Hence either of the last two
conditions implies the first.
\end{proof}

\section{Local derivations of nilpotent Leibniz algebras}
\label{sec:leibniz-local-maps}
In this section, we prove that every finite-dimensional non-abelian
nilpotent complex right Leibniz algebra admits a proper local
derivation whenever its dimension is at most seven or its nilpotency
class is at most four. Both bounds are sharp: we exhibit an
eight-dimensional non-Lie algebra $A_8$ of class five satisfying
$\LocDer(A_8)=\Der(A_8)$. Consequently, eight and five are the
least possible dimension and nilpotency class, respectively, for
this equality among non-abelian nilpotent complex right Leibniz
algebras.

\subsection{Notation and auxiliary constructions}
We recall the notation and record the auxiliary derivation
constructions used below. Throughout this section, $A$ denotes
a finite-dimensional complex right Leibniz algebra, so
\[
[x,[y,z]]=[[x,y],z]-[[x,z],y].
\]
We use the right lower central series
\[
A^1=A,\qquad A^{j+1}=[A^j,A].
\]
This agrees with the lower central series used above when $A$ is Lie. Write
\[
d_j=\dim(A^j/A^{j+1}),\qquad\mathcal N=\{z\in A:[A,z]=0\},
\]
and let $Z(A)$ be the two-sided centre. The definitions of derivations and local derivations are
unchanged. A local derivation is called proper if it is not a derivation.

Induction in the right Leibniz identity gives
\begin{equation}\label{eq:min-filtration}
[A^i,A^j]\subseteq A^{i+j}.
\end{equation}
Every $R_z:x\mapsto[x,z]$ is a derivation. Squares and symmetric brackets lie in $\mathcal N$, and $\mathcal N$ is stable
under right multiplications: for $n\in\mathcal N$,
\[
[x,[n,z]]=[[x,n],z]-[[x,z],n]=0.
\]
If $A$ has class $c$, then $A^c\subseteq Z(A)$. Its nonzero lower central terms decrease strictly: equality of two
consecutive terms would persist indefinitely and contradict nilpotency. Thus $d_j\ge1$ for $1\le j\le c$.

\begin{lemma}\label{lem:min-projection}
Suppose $D\in\Der(A)$ satisfies
\begin{equation}\label{eq:min-K}
0\ne D(A^2)\subseteq Z(A).
\end{equation}
For a complement $A=V\oplus A^2$, let $\pi$ be the
projection onto $A^2$. Then $\Delta=D\pi$ is a proper
local derivation.
\end{lemma}

\begin{proof}
The proof of Lemma~\ref{lem:central-criterion-local}
applies verbatim to right Leibniz algebras, since it uses
only the derivation identity, the two-sided centrality
of $D(A^2)$, and elementary linear algebra.
\end{proof}
\begin{lemma}\label{lem:min-ann}
If $A$ is nilpotent and $A^2\subseteq\mathcal N+A^3$, then $[A,A^2]=0$. Every non-abelian nilpotent
algebra satisfying $[A,A^2]=0$ admits a proper local derivation. In particular, this holds if $d_1=1$,
or if $d_2=1$ and some square does not belong to $A^3$.
\end{lemma}
\begin{proof}
Stability of $\mathcal N$ under right multiplication implies
\[
A^j\subseteq\mathcal N+A^{j+1}\ \Longrightarrow\ A^{j+1}\subseteq\mathcal N+A^{j+2}.
\]
Starting with $j=2$ gives $A^2\subseteq\mathcal N+A^j$ for every $j\ge3$. Nilpotency gives $A^2\subseteq\mathcal N$.

Assume $[A,A^2]=0$ and let $c$ be the nilpotency class. If $c=2$, choose $A=V\oplus A^2$ and take $D$
equal to the identity on $V$ and twice the identity on $A^2$. It satisfies~\eqref{eq:min-K}. If $c\ge3$, every nonempty word in right multiplications is a derivation. Indeed, for one right multiplication the derivation rule
reduces to
\[
R_z[a,b]=[R_za,b],
\]
because $R_zb\in A^2\subseteq\mathcal N$. Iterating gives the same identity for a word $D$, while $D(b)\in A^2$ makes
$[a,D(b)]=0$. Choose a word of length $c-2$ whose restriction to $A^2$ is nonzero; such a word exists
because $A^c$ is spanned by iterated right multiplications of brackets. It maps $A^2$ into $A^c\subseteq Z(A)$,
so Lemma~\ref{lem:min-projection} applies.

If $d_1=1$, write $A=\C e+A^2$. Expansion modulo $A^3$ gives $A^2\subseteq\C[e,e]+A^3\subseteq\mathcal N+A^3$. If $d_2=1$
and $[x,x]\notin A^3$, then $A^2=\C[x,x]+A^3$, and the same argument applies.
\end{proof}
\begin{lemma}\label{lem:min-split}
Let $m\ge1$, $A^{m+3}=0\ne A^{m+2}$, and suppose
\[
[A^{i+1},A^{j+1}]=0\qquad(i,j\ge1,\ i+j=m).
\]
Then every word of $m$ right multiplications is a derivation, and $A$ admits a proper local derivation.
\end{lemma}
\begin{proof}
Let $C_k$ be the sum of $[A^{i+1},A^{j+1}]$ over $i,j\ge1$ with $i+j=k$; set $C_1=0$. The derivation
rule for $R_y$ gives $R_y(C_k)\subseteq C_{k+1}$. Inductively, the defect of a word $D$ of length $k$ lies in $C_k$: the
assertion is true for $k=1$, and the defect of $R_yD$ is
\[
[Da,R_yb]+[R_ya,Db]+R_y\bigl(D[a,b]-[Da,b]-[a,Db]\bigr),
\]
which belongs to $C_{k+1}$. Since $C_m=0$, a word of length $m$ is a derivation. Such words send $A^2$ into
$A^{m+2}\subseteq Z(A)$. Some word is nonzero on a bracket because $A^{m+2}\ne0$. Apply Lemma~\ref{lem:min-projection}.
\end{proof}
\begin{proposition}\label{prop:min-smallclass}
Every non-abelian nilpotent right Leibniz algebra of class at most four admits a
proper local derivation. An algebra of class five does also, unless
\begin{equation}\label{eq:min-Y}
[A^2,A^3]=0,\qquad[A^3,A^2]\ne0.
\end{equation}
\end{proposition}
\begin{proof}
The class-two case was proved in Lemma~\ref{lem:min-ann}. In class three, some $R_y$ is nonzero on $A^2$, and
its image there lies in $A^3\subseteq Z(A)$. In class four, if $[A^2,A^2]\ne0$, choose $w\in A^2$ with $R_w(A^2)\ne0$.
Otherwise apply Lemma~\ref{lem:min-split} with $m=2$. In class five, if $[A^2,A^3]\ne0$, choose $w\in A^3$ with
$R_w(A^2)\ne0$. Its image lies in $A^5\subseteq Z(A)$. If both $[A^2,A^3]$ and $[A^3,A^2]$ vanish, apply Lemma~\ref{lem:min-split}
with $m=3$.
\end{proof}

\subsection{Excluding the small class-five cases}
We now handle the only possible obstruction in
class five. The first lemma uses one-dimensional third and fourth lower central factors, with no
restriction on the number of generators.
\begin{lemma}\label{lem:min-T}
Suppose $A^6=0\ne A^5$ and $d_3=d_4=1$. If $[A^2,A^3]=0$, then $[A^3,A^2]=0$.
\end{lemma}
\begin{proof}
Choose $u\in A^3$, $w\in A^4$ with $A^3=\C u+A^4$ and $A^4=\C w+A^5$. Since $[A^4,A^2]=0$, it
suffices to show $[u,A^2]=0$. Suppose instead that $n=[u,p]\ne0$ for some bracket $p=[a,b]$; such a
bracket exists by linearity. Write
\[
[u,a]=c_aw+r_a,\qquad[u,b]=c_bw+r_b,\qquad r_a,r_b\in A^5.
\]
The Leibniz identity gives
\begin{equation}\label{eq:min-n}
n=c_a[w,b]-c_b[w,a].
\end{equation}
Fix $y\in A^2$, and write $[y,a]=\alpha_au+\ell_a$, $[y,b]=\alpha_bu+\ell_b$, where $\ell_a,\ell_b\in A^4$. Put $k=\alpha_ac_b-\alpha_bc_a$.

For arbitrary $x\in A$, write $[y,x]=\alpha u+\ell$, with $\ell\in A^4$. The hypothesis gives $[y,[p,x]]=0$, hence
\[
[[y,p],x]=[[y,x],p]=\alpha n.
\]
On the other hand, expand $[y,p]=[[y,a],b]-[[y,b],a]$. Modulo $A^5$ it is $kw$, so
\begin{equation}\label{eq:min-k}
k[w,x]=\alpha n.
\end{equation}

Applying this at $x=a,b$ and using~\eqref{eq:min-n} yields
\[
kn=(c_a\alpha_b-c_b\alpha_a)n=-kn.
\]
Since $n\ne0$, we obtain $k=0$. Equation~\eqref{eq:min-k} then forces $\alpha=0$ for every $x$. As $y\in A^2$ was
arbitrary, $A^3=[A^2,A]\subseteq A^4$, contradicting nilpotency and $A^5\ne0$.
\end{proof}
The remaining lemmas concern two generators. Suppose
\begin{equation}\label{eq:min-alt}
A=\C e_1+\C e_2+A^2,\qquad[x,x]\in A^3\quad(x\in A).
\end{equation}
Set $p=[e_1,e_2]$, $s_i=[e_i,e_i]$, and $t=[e_1,e_2]+[e_2,e_1]$. Polarization gives $s_1,s_2,t\in A^3$. Expanding
brackets modulo the next lower central term gives
\begin{equation}\label{eq:min-decomp}
A^2=\C p+A^3,\qquad A^3=\C[p,e_1]+\C[p,e_2]+A^4.
\end{equation}
For example, if $x=ae_1+be_2+r$ and $y=a'e_1+b'e_2+r'$ with $r,r'\in A^2$, then $[x,y]\equiv(ab'-ba')p$
$(\bmod\ A^3)$, proving the first equality. The second follows by expanding $[cp+r_3,ae_1+be_2+r_2]$
modulo $A^4$.
\begin{lemma}\label{lem:min-V}
Assume~\eqref{eq:min-alt}, $A^6=0\ne A^5$ and $d_4=1$. If $[A^2,A^3]=0$, then $[A^3,A^2]=0$.
\end{lemma}
\begin{proof}
Put $u_i=[p,e_i]$. From $[p,[p,e_i]]=0$ we get
\begin{equation}\label{eq:min-ui}
[u_i,p]=[[p,p],e_i].
\end{equation}
If $[p,p]\in A^5$, these brackets vanish. The decompositions~\eqref{eq:min-decomp} and~\eqref{eq:min-filtration} then imply $[A^3,A^2]=0$.

Suppose $[p,p]\notin A^5$. Since $d_4=1$, $A^4=\C[p,p]+A^5$. A square belongs to $\mathcal N$, and $A^5\subseteq Z(A)$;
hence $[A,A^4]=0$. For $v\in A^3$ the Leibniz identity consequently gives
\[
[[e_i,v],e_j]=[[e_i,e_j],v]=0\qquad(i,j=1,2).
\]
The last equality uses $[A^2,A^3]=0$. Write $c=[u_1,e_2]$, $d=[u_2,e_1]$. The exact identities
\[
[e_1,p]=[s_1,e_2]-u_1,\quad[e_2,p]=[t,e_2]-u_2-[s_2,e_1],\quad[p,p]=c-d
\]
give, modulo $A^5$,
\[
\begin{aligned}
{}[e_1,u_1]&\equiv-[u_1,e_1],&[e_2,u_2]&\equiv-[u_2,e_2],\\
[e_1,u_2]&\equiv-2c+d,&[e_2,u_1]&\equiv c-2d.
\end{aligned}
\]
Indeed $[e_i,u_j]=[[e_i,p],e_j]-[[e_i,e_j],p]$, and the omitted terms belong to $A^5$. Right multiplication
annihilates these error terms. Thus, for $k=1,2$,
\[
[[u_1,e_1],e_k]=[[u_2,e_2],e_k]=[-2c+d,e_k]=[c-2d,e_k]=0.
\]
The coefficient matrix with rows $(-2,1)$ and $(1,-2)$ has determinant three. Therefore $[c,e_k]=[d,e_k]=0$.
We have proved $[[u_i,e_j],e_k]=0$ for all $i,j,k$. Replacing $e_j,e_k$ by arbitrary elements
introduces only terms in $A^6=0$. Using~\eqref{eq:min-decomp} gives $[[A^3,A],A]=0$, or $A^5=0$, a contradiction.
\end{proof}
\begin{lemma}\label{lem:min-d3}
Assume~\eqref{eq:min-alt} and $d_3=1$. Generators can be chosen so that, with $p=[e_1,e_2]$ and
$u=[p,e_1]$,
\[
[p,e_2]\in A^4,\qquad A^3=\C u+A^4.
\]
For these generators $[u,e_2]\in A^5$ and $A^4=\C[u,e_1]+A^5$. In particular $d_4\le1$.
\end{lemma}
\begin{proof}
Choose $u_0\in A^3\setminus A^4$, and write $[p,e_1]=\alpha u_0+r_1$, $[p,e_2]=\beta u_0+r_2$, where $r_i\in A^4$. At least
one of $\alpha,\beta$ is nonzero by~\eqref{eq:min-decomp}. If $\alpha\ne0$, replace the generators by
\[
f_1=e_1,\qquad f_2=\alpha e_2-\beta e_1.
\]
This is an invertible change modulo $A^2$. Since $[f_1,f_2]=\alpha p-\beta s_1$, one obtains
\[
[[f_1,f_2],f_2]\in A^4,\qquad [[f_1,f_2],f_1]\equiv\alpha^2u_0\pmod{A^4}.
\]

If $\alpha=0$, then $\beta\ne0$; interchange $e_1,e_2$ and use $[e_2,e_1]=t-p$. The corresponding two brackets are
in $A^4$ and congruent to $-\beta u_0$ modulo $A^4$, respectively. Thus this case gives the same normalization.
Rename the new generators $e_1,e_2$.

Now $[e_1,[p,e_2]]\in A^5$. The Leibniz identity and $[e_1,p]=[s_1,e_2]-u$ give
\[
[e_1,[p,e_2]]=[[e_1,p],e_2]-[p,p]\equiv-[u,e_2]-[p,p]\pmod{A^5}.
\]
Also $[p,p]=[u,e_2]-[[p,e_2],e_1]\equiv[u,e_2]\pmod{A^5}$. Hence $2[u,e_2]\in A^5$. Finally expand the
generators of $[A^3,A]$ using $A^3=\C u+A^4$ and the two generators modulo $A^2$. All terms except
multiples of $[u,e_1]$ lie in $A^5$, proving the last assertion.
\end{proof}

\subsection{The remaining class-six case}
The dimension bound will leave one class-six dimension
vector. The following calculation handles it inside the original algebra and constructs the required
derivation there.
\begin{lemma}\label{lem:min-six}
Suppose~\eqref{eq:min-alt}, $A^7=0\ne A^6$ and
\[
(d_1,d_2,d_3,d_4,d_5,d_6)=(2,1,1,1,1,1).
\]
Then $A$ admits a proper local derivation.
\end{lemma}
\begin{proof}
Normalize the generators by Lemma~\ref{lem:min-d3}. Put $p=[e_1,e_2]$, $u=[p,e_1]$, $w=[u,e_1]$, and choose
$z\in A^5$ with $A^5=\C z+A^6$. We have
\[
A^2=\C p+A^3,\quad A^3=\C u+A^4,\quad A^4=\C w+A^5,\quad[u,e_2]\in A^5.
\]
Define
\[
P=[[w,e_2],e_1],\qquad Q=[[w,e_1],e_2].
\]
All terms of total lower central degree at least seven vanish. From
\[
[u,[p,e_2]]=[[u,p],e_2]-[[u,e_2],p]
\]
and $[u,p]=[w,e_2]-[[u,e_2],e_1]$, we obtain
\begin{equation}\label{eq:min-six-zero}
[[w,e_2],e_2]=0.
\end{equation}
Moreover,
\begin{equation}\label{eq:min-six-PQ}
[w,p]=Q-P,\qquad[u,u]=2P-Q,\qquad[p,w]=Q-3P.
\end{equation}
The first equality is the Leibniz identity. For the second, expand $[u,[p,e_1]]$ and note $[[u,p],e_1]=P$.
For the third, expand $[p,[u,e_1]]$. Here $[p,u]=[[p,p],e_1]-[u,p]$ and $[p,p]\in A^5$ by Lemma~\ref{lem:min-d3};
consequently $[[p,u],e_1]=-P$.

If $[p,w]\ne0$, the derivation $R_w$ maps $A^2$ into $A^6\subseteq Z(A)$ and is nonzero there. Apply Lemma~\ref{lem:min-projection}.

Suppose instead that $[p,w]=0$, so $Q=3P$. Write $[w,e_2]=\beta z+r$ with $r\in A^6$. If $\beta=0$,
then $P=0$ and $Q=0$. If $\beta\ne0$, \eqref{eq:min-six-zero} gives $[z,e_2]=0$. Since $[w,e_1]\in\C z+A^6$, this implies
$Q=0$, hence $P=0$. Now $P=\beta[z,e_1]$ gives $[z,e_1]=0$. Products of $z\in A^5$ with $A^2$ vanish, and
$[A^6,A]=0$. Thus $[A^5,A]=0$, contradicting $A^6\ne0$. This excludes $\beta\ne0$, and in all remaining
cases $P=Q=0$.

By~\eqref{eq:min-six-PQ}, $[p,w]=[u,u]=[w,p]=0$. The displayed decompositions then imply
\[
[A^2,A^4]=[A^3,A^3]=[A^4,A^2]=0;
\]
all other contributions have degree at least seven. Lemma~\ref{lem:min-split} with $m=4$ completes the proof.
\end{proof}

\subsection{The lower bound on the dimension}
The preceding structural lemmas now yield an exhaustive
argument in dimensions at most seven.
\begin{theorem}\label{thm:min-small}
Every finite-dimensional non-abelian nilpotent complex right Leibniz algebra $A$ with
$\dim A\le7$ admits a proper local derivation.
\end{theorem}
\begin{proof}
If the nilpotency class is at most four, use Proposition~\ref{prop:min-smallclass}. If $d_1=1$, use Lemma~\ref{lem:min-ann}. We
can therefore assume that $c\ge5$ and $d_1\ge2$.

For $c=5$, strict decrease of the lower central series and $\sum_{i=1}^5d_i\le7$ leave exactly
\[
\begin{gathered}
(2,1,1,1,1),\ (3,1,1,1,1),\ (2,2,1,1,1),\\
(2,1,2,1,1),\ (2,1,1,2,1),\ (2,1,1,1,2).
\end{gathered}
\]
By Proposition~\ref{prop:min-smallclass}, it remains only to exclude~\eqref{eq:min-Y}. For the first, second, third and sixth vectors,
$d_3=d_4=1$, so Lemma~\ref{lem:min-T} excludes it.

For the two remaining vectors $d_1=2$ and $d_2=1$. If some square lies outside $A^3$, Lemma~\ref{lem:min-ann}
applies. Otherwise choose two generators modulo $A^2$ and~\eqref{eq:min-alt} holds. For $(2,1,2,1,1)$, Lemma~\ref{lem:min-V}
excludes~\eqref{eq:min-Y}. For $(2,1,1,2,1)$, Lemma~\ref{lem:min-d3} gives $d_4\le1$, contradicting $d_4=2$.

For $c\ge6$, the inequality $\dim A\ge2+(c-1)$ implies $c=6$, $\dim A=7$, and the vector
$(2,1,1,1,1,1)$. A square outside $A^3$ again gives Lemma~\ref{lem:min-ann}; otherwise apply Lemma~\ref{lem:min-six}. This
treats every possible dimension vector.
\end{proof}

\subsection{An eight-dimensional counterexample}
We next show that the dimension bound is sharp.
Let $A_8$ have basis $e_1,\ldots,e_8$ and the following nonzero products, with all omitted products zero:
\begin{equation}\label{eq:min-A8}
\begin{aligned}
{}[e_1,e_2]&=-e_3,&[e_2,e_1]&=e_3,&[e_2,e_2]&=e_4,\\
[e_2,e_3]&=e_5,&[e_4,e_1]&=e_5,&[e_4,e_2]&=e_6,\\
[e_4,e_3]&=e_7,&[e_5,e_2]&=e_7,&[e_6,e_1]&=2e_7,\\
[e_6,e_2]&=e_7,&[e_5,e_3]&=3e_8,&[e_6,e_3]&=e_8,\\
[e_7,e_1]&=3e_8,&[e_7,e_2]&=e_8.
\end{aligned}
\end{equation}
\begin{proposition}\label{prop:min-A8}
The algebra $A_8$ is a non-abelian non-Lie nilpotent right Leibniz algebra of class
five. It satisfies
\[
\LocDer(A_8)=\Der(A_8),\qquad\dim\Der(A_8)=7.
\]
\end{proposition}
\begin{proof}
We give a direct proof of the Leibniz identity, determine all derivations, and then show that
locality forces the same formulas.

\emph{The algebra and its lower central series.} Let $R_i(x)=[x,e_i]$. From~\eqref{eq:min-A8}, the only nonzero right
multiplications are $R_1,R_2,R_3$, whose nonzero values are
\[
\begin{array}{ll}
R_1& e_2\mapsto e_3,\ e_4\mapsto e_5,\ e_6\mapsto2e_7,\ e_7\mapsto3e_8\\
R_2& e_1\mapsto-e_3,\ e_2\mapsto e_4,\ e_4\mapsto e_6,\ e_5\mapsto e_7,\ e_6\mapsto e_7,\ e_7\mapsto e_8\\
R_3& e_2\mapsto e_5,\ e_4\mapsto e_7,\ e_5\mapsto3e_8,\ e_6\mapsto e_8.
\end{array}
\]
Here $[S,T]=ST-TS$ for linear operators. The displayed actions give
\[
[R_1,R_2]=R_3,\qquad[R_1,R_3]=[R_2,R_3]=0.
\]
For example, $[R_1,R_2]$ sends $e_2,e_4,e_5,e_6$ to $e_5,e_7,3e_8,e_8$, respectively, and kills the other basis
vectors. These relations, together with $R_i=0$ for $i\ge4$, give
\[
R_{[e_i,e_j]}=[R_j,R_i]\qquad(1\le i,j\le8).
\]

By bilinearity, $R_{[y,z]}=R_zR_y-R_yR_z$ for all $y,z$. Applied to $x$, this is exactly the right Leibniz
identity. The lower central series follows immediately from the multiplication table:
\[
\begin{gathered}
A_8^2=\langle e_3,e_4,e_5,e_6,e_7,e_8\rangle,\\
A_8^3=\langle e_5,e_6,e_7,e_8\rangle,\quad A_8^4=\langle e_7,e_8\rangle,\quad
A_8^5=\langle e_8\rangle,\quad A_8^6=0.
\end{gathered}
\]
The nonzero square $[e_2,e_2]=e_4$ shows that $A_8$ is non-abelian and is not a Lie algebra.

\emph{Determination of the derivations.} The full derivation family, with arbitrary $a,b,p,q,r,s,t\in\C$,
is
\begin{equation}\label{eq:min-der-family}
\begin{aligned}
De_1&=ae_1+pe_3+se_8,\\
De_2&=be_1+(a-3b)e_2+qe_3-pe_4+re_5+te_8,\\
De_3&=(2a-3b)e_3,\\
De_4&=(2a-6b)e_4+qe_5-pe_6+re_7,\\
De_5&=(3a-6b)e_5-pe_7+3re_8,\\
De_6&=be_5+(3a-9b)e_6+(2q-p)e_7+re_8,\\
De_7&=(4a-9b)e_7+(3q-p)e_8,\\
De_8&=(5a-9b)e_8.
\end{aligned}
\end{equation}
To prove completeness, let $D$ be a derivation and write
\[
De_1=\sum_{i=1}^8u_ie_i,\qquad De_2=\sum_{i=1}^8v_ie_i.
\]
The algebra is generated by $e_1,e_2$, since
\[
e_3=[e_2,e_1],\ e_4=[e_2,e_2],\ e_5=[e_2,e_3],\ e_6=[e_4,e_2],\ e_7=[e_5,e_2],\ e_8=[e_6,e_3].
\]
Thus the derivation rule determines $De_3,\ldots,De_8$ from $De_1,De_2$. Applying it to $[e_1,e_1]=0$ gives
\[
u_4e_5+2u_6e_7+3u_7e_8=0,
\]
so $u_4=u_6=u_7=0$. Applying it to $[e_1,e_3]=0$ then gives $u_2e_5+3u_5e_8=0$, hence $u_2=u_5=0$.
Before the remaining coefficients have been simplified, the generating products give
\[
\begin{aligned}
De_3&=(u_1+v_2)e_3+(u_3+v_4)e_5+2v_6e_7+3v_7e_8,\\
De_4&=2v_2e_4+v_3e_5+v_4e_6+(v_5+v_6)e_7+v_7e_8,\\
De_5&=(u_1+2v_2)e_5+v_4e_7+(3v_5+v_6)e_8,\\
De_6&=v_1e_5+3v_2e_6+(2v_3+v_4)e_7+(v_5+v_6)e_8,\\
De_7&=(u_1+3v_2)e_7+(3v_3+v_4)e_8,\\
De_8&=(u_1+3v_1+4v_2)e_8.
\end{aligned}
\]
The zero product $[e_3,e_1]=0$ gives $6v_6e_8=0$, so $v_6=0$. Next, $[e_3,e_2]=0$ gives $u_3+v_4=0$, and
$[e_1,e_2]=-e_3$ gives $v_7=0$. Finally, compare the derivation rule for $[e_6,e_2]=e_7$ with the displayed
expression for $De_7$. The $e_7$ coordinate gives
\[
u_1=3v_1+v_2.
\]
Setting $a=u_1$, $b=v_1$, $p=u_3$, $q=v_3$, $r=v_5$, $s=u_8$, $t=v_8$ yields exactly~\eqref{eq:min-der-family}.

Conversely, a linear map in~\eqref{eq:min-der-family} satisfies
\[
\begin{aligned}
{}[D,R_1]&=aR_1+pR_3,\\
[D,R_2]&=bR_1+(a-3b)R_2+qR_3,\\
[D,R_3]&=(2a-3b)R_3.
\end{aligned}
\]

These equalities follow by applying both sides to the basis using the three right multiplications
listed above. Their right sides are $R_{De_1},R_{De_2},R_{De_3}$, respectively. For $i\ge4$, both $[D,R_i]$ and $R_{De_i}$
are zero. Thus $[D,R_y]=R_{Dy}$ for every $y$, which is equivalent to the derivation rule. This proves
the converse. The seven parameters are independent, so $\dim\Der(A_8)=7$.

\emph{Normalization of a local derivation.} Let $T$ be a local derivation. Choose $D_0\in\Der(A_8)$
with
\[
T(e_1+e_2)=D_0(e_1+e_2)
\]
and put $\Delta=T-D_0$. Subtracting a derivation preserves locality: a witness $D_x$ for $T$ is replaced
by $D_x-D_0$. Hence $\Delta$ is local and $\Delta(e_1+e_2)=0$. Locality at $e_1$ gives
\[
\Delta(e_1)=a_0e_1+ce_3+de_8,\qquad\Delta(e_2)=-a_0e_1-ce_3-de_8.
\]
For any derivation in~\eqref{eq:min-der-family}, the first two coordinates of $D(e_1-3e_2)$ are $(a-3b,-3(a-3b))$. The
first two coordinates of $\Delta(e_1-3e_2)=4\Delta(e_1)$ are $(4a_0,0)$. A local witness therefore forces $a_0=0$.
Consequently
\begin{equation}\label{eq:min-normal-gen}
\Delta(e_1)=ce_3+de_8,\qquad\Delta(e_2)=-ce_3-de_8.
\end{equation}
Locality at the remaining basis vectors gives scalars such that
\begin{equation}\label{eq:min-normal-cols}
\begin{aligned}
\Delta(e_3)&=\ell e_3,\\
\Delta(e_4)&=t_0e_4+ue_5+ve_6+we_7,\\
\Delta(e_5)&=me_5+ne_7+he_8,\\
\Delta(e_6)&=re_5+s_0e_6+je_7+ke_8,\\
\Delta(e_7)&=fe_7+ge_8,\qquad\Delta(e_8)=ze_8.
\end{aligned}
\end{equation}
In particular, the first two coordinates of $\Delta(x)$ vanish for every $x$.

\emph{Relations from a local witness at a general vector.} Fix $x=\sum_{i=1}^8x_ie_i$ with
\begin{equation}\label{eq:min-open}
x_2\ne0,\qquad3x_1+x_2\ne0,
\end{equation}
and choose a derivation $D_x$ satisfying $\Delta(x)=D_x(x)$. Denote its parameters in~\eqref{eq:min-der-family} by
$a_x,b_x,p_x,q_x,r_x,s_x,t_x$. For every derivation of that form,
\[
(D_xx)_2=(a_x-3b_x)x_2,\qquad3(D_xx)_1+(D_xx)_2=a_x(3x_1+x_2).
\]
Since the left sides vanish, \eqref{eq:min-open} implies $a_x=b_x=0$. Writing $y_i$ for the coordinates of $\Delta(x)$, the
remaining formulas give
\begin{equation}\label{eq:min-witness}
\begin{aligned}
y_3&=p_xx_1+q_xx_2,&y_4&=-p_xx_2,\\
y_5&=r_xx_2+q_xx_4,&y_6&=-p_xx_4,\\
y_7&=r_xx_4-p_xx_5+(2q_x-p_x)x_6.
\end{aligned}
\end{equation}
In particular, $x_2y_6-x_4y_4=0$. Substituting~\eqref{eq:min-normal-cols}, we obtain
\[
vx_2x_4+s_0x_2x_6-t_0x_4^2=0.
\]
This polynomial identity holds whenever~\eqref{eq:min-open} holds, and hence holds identically. Indeed the
excluded set is the union of two hyperplanes, so its complement is dense, and a polynomial is
continuous. Comparing the coefficients of the three distinct monomials gives
\[
v=s_0=t_0=0.
\]
Thus $y_4=y_6=0$, and the equation $y_4=-p_xx_2$ gives $p_x=0$ at every vector satisfying~\eqref{eq:min-open}. The
remaining equations in~\eqref{eq:min-witness} reduce to
\[
y_3=q_xx_2,\qquad y_5=r_xx_2+q_xx_4,\qquad y_7=r_xx_4+2q_xx_6.
\]

Eliminating $q_x,r_x$ gives the explicit relation
\begin{equation}\label{eq:min-poly}
(2x_2x_6-x_4^2)y_3+x_2x_4y_5-x_2^2y_7=0.
\end{equation}
By~\eqref{eq:min-normal-gen} and~\eqref{eq:min-normal-cols}, this is
\[
\begin{aligned}
0={}&(2x_2x_6-x_4^2)\bigl(c(x_1-x_2)+\ell x_3\bigr)\\
&+x_2x_4(ux_4+mx_5+rx_6)-x_2^2(wx_4+nx_5+jx_6+fx_7).
\end{aligned}
\]
Again it holds identically as a polynomial. The coefficients of $x_1x_4^2$ and $x_3x_4^2$ are $-c$ and $-\ell$, so
$c=\ell=0$. The identity then becomes
\[
ux_2x_4^2+mx_2x_4x_5+rx_2x_4x_6-wx_2^2x_4-nx_2^2x_5-jx_2^2x_6-fx_2^2x_7=0.
\]
Its seven monomials are distinct. Their coefficients must vanish, giving
\[
u=m=r=w=n=j=f=0.
\]
Consequently
\[
\Delta(e_3)=\Delta(e_4)=0,\quad\Delta(e_5)=he_8,\quad\Delta(e_6)=ke_8,\quad
\Delta(e_7)=ge_8,\quad\Delta(e_8)=ze_8.
\]
\emph{The central coefficients.} For each $i\in\{5,6,7,8\}$, put $w_i=e_3+e_4+e_i$. Then $\Delta(w_i)$ is a multiple
of $e_8$. Let $D_i$ be a local witnessing derivation with $D_i(w_i)=\Delta(w_i)$, and denote its parameters by
$a_i,b_i,p_i,q_i,r_i,s_i,t_i$. The $e_3,e_4$ coordinates give
\[
2a_i-3b_i=0,\qquad2a_i-6b_i=0,
\]
hence $a_i=b_i=0$. The $e_5$ coordinate of $D_i(w_i)$ is now $q_i$, so $q_i=0$. Its $e_6$ coordinate is $-p_i$,
so $p_i=0$. Its $e_7$ coordinate is then $r_i$, so $r_i=0$. Since $w_i$ has zero $e_1,e_2$ coordinates, the only
parameters still available, $s_i,t_i$, contribute nothing to $D_i(w_i)$. Hence $D_i(w_i)=0$ for each $i$, and
therefore
\[
h=k=g=z=0.
\]
We have proved that $\Delta$ sends $e_1$ to $de_8$, $e_2$ to $-de_8$, and kills $e_3,\ldots,e_8$. This is the derivation in
\eqref{eq:min-der-family} with $a=b=p=q=r=0$, $s=d$ and $t=-d$. Thus $T=D_0+\Delta$ is a derivation. Conversely
every derivation is local, proving $\LocDer(A_8)=\Der(A_8)$.
\end{proof}
\begin{theorem}\label{thm:min-eight}
The least dimension of a non-abelian nilpotent complex right Leibniz algebra $A$
satisfying $\LocDer(A)=\Der(A)$ is eight. The least possible nilpotency class of such an algebra is
five.
\end{theorem}
\begin{proof}
Theorem~\ref{thm:min-small} excludes dimensions at most seven, while Proposition~\ref{prop:min-A8} supplies an example
in dimension eight and class five. Proposition~\ref{prop:min-smallclass} excludes classes at most four in every dimension.
\end{proof}

\section{Local automorphisms of nilpotent Leibniz algebras}

Throughout this section, $A$ denotes a finite-dimensional complex
right Leibniz algebra. We use the notation of
Section~\ref{sec:leibniz-local-maps}.
We prove that every non-abelian nilpotent algebra $A$ admits
a proper local automorphism whenever $\dim A\le7$ or its
nilpotency class is at most four. We then determine the full
automorphism group of the eight-dimensional algebra $A_8$
of class five defined in~\eqref{eq:min-A8} and establish
$\LocAut(A_8)=\Aut(A_8)$. Consequently, eight and five are
the least possible dimension and nilpotency class, respectively,
for this equality among non-abelian nilpotent complex
right Leibniz algebras.

\subsection{A construction of proper local automorphisms}
The projection argument for local
derivations needs an additional verification in the automorphism setting: the maps witnessing
locality must preserve products and be invertible. The following lemma supplies both properties.
\begin{lemma}\label{lem:aut-projection}
Suppose $A$ admits a derivation $D$ satisfying
\begin{equation}\label{eq:aut-strong}
D(A)\subseteq A^2,\qquad0\ne D(A^2)\subseteq Z(A),\qquad D(Z(A))=0.
\end{equation}
Let $A=V\oplus A^2$, and let $\pi$ be the projection onto $A^2$. For every $\lambda\in\C\setminus\{0\}$, the map
\[
T=\Id+\lambda D\pi
\]
is a proper local automorphism. If $A$ is non-abelian of class two, then $T=\Id+3\pi$ is a proper local
automorphism.
\end{lemma}
\begin{proof}
Assume first~\eqref{eq:aut-strong}. A derivation preserves $A^2$, since it sends every bracket to a sum of
brackets. Consequently
\[
D^2(A)\subseteq D(A^2)\subseteq Z(A)\cap A^2,\qquad D^3=0,\qquad D^2(A^2)=0.
\]
Set $N=D\pi$. Its image lies in $Z(A)\cap A^2$, on which $D$ vanishes and $\pi$ acts as the identity. Thus
\[
N^2(x)=D\pi(D\pi x)=D^2\pi x=0,
\]
and $T^{-1}=\Id-\lambda N$.

We record explicitly that these hypotheses also imply $[D(A),D(A)]=0$. Applying the derivation
identity twice gives
\[
D^2[u,v]=[D^2u,v]+2[Du,Dv]+[u,D^2v].
\]
The left side vanishes because $[u,v]\in A^2$, and the first and last terms vanish because $D^2(A)\subseteq Z(A)$.
Since the characteristic is zero, $[Du,Dv]=0$. It follows that
\[
\psi=\Id+\lambda D
\]
preserves every bracket:
\[
[\psi u,\psi v]=[u,v]+\lambda D[u,v]+\lambda^2[Du,Dv]=\psi[u,v].
\]
Moreover $(\Id+\lambda D)^{-1}=\Id-\lambda D+\lambda^2D^2$, because $D^3=0$. Hence $\psi$ is an automorphism. If $x\in A^2$,
then $\pi x=x$ and $\psi(x)=T(x)$.

If $x\notin A^2$, put $z=\lambda D\pi x\in Z(A)\cap A^2$. Choose $f\in A^*$ with $f(A^2)=0$ and $f(x)=1$. Define
\[
\sigma_x(y)=y+f(y)z.
\]
Centrality of $z$ and $f(A^2)=0$ give
\[
[\sigma_xu,\sigma_xv]=[u,v]=\sigma_x[u,v].
\]
Since $f(z)=0$, its inverse is $y\mapsto y-f(y)z$. Thus $\sigma_x$ is an automorphism and $\sigma_x(x)=x+z=T(x)$.
This proves locality at every vector.

Finally, $N(A)\subseteq Z(A)$ implies $[Tu,Tv]=[u,v]$, whereas
\[
T[u,v]=[u,v]+\lambda D[u,v].
\]
The space $A^2$ is spanned by brackets and $D(A^2)\ne0$, so some pair $u,v$ has $D[u,v]\ne0$. For this
pair the two expressions differ. Therefore $T$ is proper.

Now suppose $A$ has class two. The filtration inclusion~\eqref{eq:min-filtration} gives both $[A^2,A]=0$ and $[A,A^2]=0$,
so $A^2\subseteq Z(A)$. The map $T=\Id+3\pi$ is invertible, with inverse $\Id-\tfrac34\pi$. The map $\psi=2\Id+2\pi$
multiplies $V$ by two and $A^2$ by four; centrality of $A^2$ gives
\[
[\psi u,\psi v]=4[u,v]=\psi[u,v].
\]
It is invertible and agrees with $T$ on $A^2$. At $x\notin A^2$ use the preceding shear construction with
$z=3\pi x$. Thus $T$ is local. For any nonzero bracket,
\[
T[u,v]=4[u,v]\ne[u,v]=[Tu,Tv],
\]
which proves properness in this case as well.
\end{proof}
\begin{remark}
Under~\eqref{eq:aut-strong}, one may alternatively use $\exp(\lambda D)=\Id+\lambda D+\tfrac{\lambda^2}{2}D^2$ on $A^2$. The
preceding proof shows that the simpler witness $\Id+\lambda D$ already works: the vanishing of $[D(A),D(A)]$
follows from the stated hypotheses, rather than being an additional assumption.
\end{remark}
\begin{theorem}\label{thm:aut-small}
Every non-abelian nilpotent complex right Leibniz algebra of dimension at most
seven admits a proper local automorphism. The same conclusion holds in every finite dimension if
the class is at most four, and in class five unless $[A^2,A^3]=0\ne[A^3,A^2]$.
\end{theorem}
\begin{proof}
We explain why the structural constructions of Section~5 give the stronger hypotheses~\eqref{eq:aut-strong},
rather than merely the existence of a proper local derivation. If a nonempty word
\[
D=R_{y_r}\cdots R_{y_1},\qquad r\ge1,
\]
is a derivation and satisfies $0\ne D(A^2)\subseteq Z(A)$, then
\[
D(A)\subseteq A^{r+1}\subseteq A^2,\qquad D(Z(A))=0.
\]
The second equality holds because the rightmost factor $R_{y_1}$, which acts first, kills the two-sided centre. Thus every such
word satisfies~\eqref{eq:aut-strong}.

For completeness, we list the constructions used in the dimension argument. Class two is covered
directly by Lemma~\ref{lem:aut-projection}. In class three, choose $y$ with $[A^2,y]\ne0$ and put $D=R_y$. In class
four, either choose $w\in A^2$ with $[A^2,w]\ne0$ and put $D=R_w$, or, when $[A^2,A^2]=0$, use the
nonzero word of length two from Lemma~\ref{lem:min-split}. In class five, if $[A^2,A^3]\ne0$, take $D=R_w$ for
an appropriate $w\in A^3$; if both $[A^2,A^3]$ and $[A^3,A^2]$ vanish, use a word of length three from that
lemma. Every displayed derivation is nonzero on $A^2$ and sends it into the last nonzero lower central
term. Lemma~\ref{lem:aut-projection} proves all assertions concerning the nilpotency class.

If $[A,A^2]=0$ and the class $c\ge3$, Lemma~\ref{lem:min-ann} constructs a word of length $c-2$ with $0\ne D(A^2)\subseteq
A^c\subseteq Z(A)$. Hence this case too satisfies~\eqref{eq:aut-strong}. In particular, it includes $d_1=1$, and $d_2=1$ when
a square lies outside $A^3$.

Assume now $\dim A\le7$ and that these cases have not applied.  As in the proof of Theorem~\ref{thm:min-small},
only classes five and six remain, with $d_1\ge2$. In class five the possible dimension vectors are
$(2,1,1,1,1),\ (3,1,1,1,1),\ (2,2,1,1,1),
(2,1,2,1,1),\ (2,1,1,2,1),\ (2,1,1,1,2).$

For the first, second, third and sixth vectors, Lemma~\ref{lem:min-T} excludes the exceptional case $[A^2,A^3]=0\ne[A^3,A^2]$.
For $(2,1,2,1,1)$, squares outside $A^3$ have already been treated; otherwise Lemma~\ref{lem:min-V}
excludes the exceptional case. For $(2,1,1,2,1)$, the same alternative on squares applies; if every
square is in $A^3$, Lemma~\ref{lem:min-d3} gives $d_4\le1$, a contradiction. Thus each class-five case has one of the
required derivations.

In class six the only remaining vector is $(2,1,1,1,1,1)$. Squares outside $A^3$ again have already
been treated. Otherwise the proof of Lemma~\ref{lem:min-six} gives either $D=R_w$ for $w\in A^4$ or a word of four
right multiplications, with $0\ne D(A^2)\subseteq A^6\subseteq Z(A)$. Both satisfy~\eqref{eq:aut-strong}. No class greater than six
is possible, because $d_1\ge2$ implies $\dim A\ge c+1$. This completes the proof of the lower dimension bound.
\end{proof}
\subsection{The full automorphism group of the eight-dimensional algebra}
We now work with
$A_8$ from~\eqref{eq:min-A8}. Computing its automorphisms explicitly will allow us to prove rigidity of local automorphisms
without drawing any inference from the equality of local derivations and derivations. To keep the
formulas readable, write the seven parameters as $a,b,\alpha,\beta,\gamma,\mu,\nu$, with $ab\ne0$.
\begin{proposition}\label{prop:aut-eight}
The automorphisms of $A_8$ are precisely the maps $\varphi$ given by
\begin{equation}\label{eq:aut-eight-family}
\begin{aligned}
\varphi(e_1)&=ae_1+a\alpha e_3+\mu e_8,\\
\varphi(e_2)&=\frac{a-b}{3}e_1+be_2+\beta e_3-\alpha be_4+\gamma e_5+\frac{\alpha^2b}{2}e_6
+\left(-\alpha\gamma+\frac{\alpha^3b}{6}\right)e_7+\nu e_8,\\
\varphi(e_3)&=abe_3,\\
\varphi(e_4)&=b^2e_4+b\left[\beta+\frac{\alpha(b-a)}{3}\right]e_5-\alpha b^2e_6+\left(\frac{a\alpha^2b}{3}-\alpha b\beta+b\gamma+\frac{\alpha^2b^2}{6}\right)e_7\\
&\quad+\left(-\frac{a\alpha^3b}{6}-a\alpha\gamma+\frac{\alpha^2b\beta}{2}+3\beta\gamma\right)e_8,\\
\varphi(e_5)&=ab^2e_5-\alpha ab^2e_7+\left(3ab\gamma+\frac{ab^2\alpha^2}{2}\right)e_8,\\
\varphi(e_6)&=\frac{ab^2-b^3}{3}e_5+b^3e_6+(2b^2\beta-\alpha ab^2)e_7\\
&\quad+\left(-2ab\alpha\beta+ab\gamma+\frac{ab^2\alpha^2}{6}+\frac{a^2b\alpha^2}{3}+3b\beta^2\right)e_8,\\
\varphi(e_7)&=ab^3e_7+(3ab^2\beta-\alpha a^2b^2)e_8,\\
\varphi(e_8)&=a^2b^3e_8.
\end{aligned}
\end{equation}
The parameters are unique, and $\det\varphi=a^6b^{15}$.
\end{proposition}
\begin{proof}
For clarity, the entire multiplication table can be written as one coordinate formula. If
$x=\sum_{i=1}^8x_ie_i$ and $y=\sum_{i=1}^8y_ie_i$, then
\begin{equation}\label{eq:aut-coordinate}
\begin{aligned}
{}[x,y]={}&(x_2y_1-x_1y_2)e_3+x_2y_2e_4+(x_2y_3+x_4y_1)e_5+x_4y_2e_6\\
&+(x_4y_3+x_5y_2+2x_6y_1+x_6y_2)e_7\\
&+(3x_5y_3+x_6y_3+3x_7y_1+x_7y_2)e_8.
\end{aligned}
\end{equation}
Let $u=\varphi(e_1)=\sum_{i=1}^8u_ie_i$ and $v=\varphi(e_2)=\sum_{i=1}^8v_ie_i$. The algebra is generated by $e_1,e_2$, as seen from
\[
e_3=[e_2,e_1],\ e_4=[e_2,e_2],\ e_5=[e_2,e_3],\ e_6=[e_4,e_2],\ e_7=[e_5,e_2],\ e_8=[e_6,e_3].
\]
Thus $u,v$ determine $\varphi$ uniquely.

The $e_4$ coordinate of $[u,u]=0$ gives $u_2^2=0$, so $u_2=0$. Since an automorphism induces an
invertible map on $A_8/A_8^2$, we have $u_1v_2\ne0$. The $e_5,e_7,e_8$ coordinates of $[u,u]=0$ now give
\[
u_4=u_6=0,\qquad u_1u_7+u_3u_5=0.
\]
Next, the $e_5,e_7$ coordinates of $[u,v]+[v,u]=0$ give
\begin{equation}\label{eq:aut-generator-rel}
u_1v_4+u_3v_2=0,\qquad2u_1v_6+u_3v_4+u_5v_2=0.
\end{equation}
The relation $[\varphi(e_3),u]=[[v,u],u]=0$ has $e_8$ coordinate
\[
6u_1^2v_6+6u_1u_3v_4+3u_3^2v_2=0.
\]
Substituting the first equality of~\eqref{eq:aut-generator-rel} yields $2u_1^2v_6=u_3^2v_2$. The second equality of~\eqref{eq:aut-generator-rel} therefore
reduces to $u_5v_2=0$. Hence $u_5=0$ and then $u_7=0$. The $e_8$ coordinate of $[u,v]+[v,u]=0$
consequently reduces to
\[
3u_1v_7=-u_3(3v_5+v_6).
\]
Finally, compute the images of $e_5,e_6$ from the generating products above. The $e_7$ coordinate of
\[
[\varphi(e_5),v]-[\varphi(e_6),v]=0
\]
is $v_2^3(u_1-3v_1-v_2)$. Therefore $u_1=3v_1+v_2$.

Set
\[
a=u_1,\quad b=v_2,\quad\alpha=\frac{u_3}{u_1},\quad\beta=v_3,\quad\gamma=v_5,\quad\mu=u_8,\quad\nu=v_8.
\]
The preceding equations give exactly the first two lines of~\eqref{eq:aut-eight-family}. Recursively applying the generating
products gives all its remaining lines.

Conversely, substitute~\eqref{eq:aut-eight-family} into~\eqref{eq:aut-coordinate}. For each pair $i,j$ one obtains
\[
[\varphi(e_i),\varphi(e_j)]=\varphi([e_i,e_j]).
\]
This is a direct polynomial verification: the nonzero right sides are the fourteen products of~\eqref{eq:min-A8},
and the other fifty right sides are zero. For example, the coordinate formula gives $[\varphi(e_2),\varphi(e_1)]=
abe_3=\varphi(e_3)$ and $[\varphi(e_5),\varphi(e_3)]=3a^2b^3e_8=3\varphi(e_8)$; applying the same formula to the other pairs
gives the indicated products and zero for every omitted product. Thus bilinearity proves product
preservation on the whole algebra. The matrix is block lower triangular with blocks indexed by
$(e_1,e_2),e_3,e_4,(e_5,e_6),e_7,e_8$. Their diagonal entries are
$a,\ b,\ ab,\ b^2,\ ab^2,\ b^3,\ ab^3,\ a^2b^3.$
Their product is $a^6b^{15}\ne0$, so $\varphi$ is invertible. The parameter definitions also prove uniqueness.
\end{proof}

\subsection{Polynomial identities for the automorphism action}
We next record the identities that
detect the nonlinear compatibility conditions between the columns of an automorphism. A polynomial
$F$ below is a semi-invariant: its value after applying $\varphi$ is multiplied by a scalar depending
only on the parameters $a,b$.
\begin{lemma}\label{lem:aut-invariants}
For $x=\sum_{i=1}^8x_ie_i$, define
\[
\begin{aligned}
\ell(x)&=3x_1+x_2,\qquad Q(x)=2x_2x_6-x_4^2,\\
P(x)&=-x_2x_4^3-3x_2x_3x_4^2+3x_2^2x_4x_6+3x_2^2x_4x_5\\
&\quad+6x_2^2x_3x_6-3x_2^3x_7-3x_1x_4^3+6x_1x_2x_4x_6.
\end{aligned}
\]
Then every automorphism in~\eqref{eq:aut-eight-family} satisfies
\begin{equation}\label{eq:aut-PQ}
\ell(\varphi x)=a\ell(x),\quad(\varphi x)_2=bx_2,\quad Q(\varphi x)=b^4Q(x),\quad P(\varphi x)=ab^6P(x).
\end{equation}
On $A_8^2$, put
\[
K(x)=x_6^3+6x_5x_6^2+9x_5^2x_6-3x_4x_6x_7-9x_4x_5x_7+3x_4^2x_8.
\]
For $x\in A_8^2$,
\begin{equation}\label{eq:aut-K}
(\varphi x)_3=abx_3,\qquad(\varphi x)_4=b^2x_4,\qquad K(\varphi x)=a^2b^7K(x).
\end{equation}
\end{lemma}
\begin{proof}
The first two identities follow immediately from the first two rows of the matrix in~\eqref{eq:aut-eight-family}.
To check $Q$, note that
\[
(\varphi x)_4=b^2x_4-\alpha bx_2,\qquad(\varphi x)_6=b^3x_6-\alpha b^2x_4+\frac{\alpha^2b}{2}x_2.
\]
Therefore
\[
2bx_2\left(b^3x_6-\alpha b^2x_4+\frac{\alpha^2b}{2}x_2\right)-(b^2x_4-\alpha bx_2)^2=b^4Q(x).
\]
For a shorter verification of $P$, rewrite it as
\[
P(x)=(\ell(x)x_4+3x_2x_3)Q(x)+x_2^2x_4(x_6+3x_5)-3x_2^3x_7.
\]
Here the additional coordinates needed are
\[
\begin{aligned}
(\varphi x)_3&=a\alpha x_1+\beta x_2+abx_3,\\
(\varphi x)_5&=\gamma x_2+b\left[\beta+\frac{\alpha(b-a)}{3}\right]x_4+ab^2x_5+\frac{ab^2-b^3}{3}x_6,\\
(\varphi x)_7&=\left(-\alpha\gamma+\frac{\alpha^3b}{6}\right)x_2
+\left(\frac{a\alpha^2b}{3}-\alpha b\beta+b\gamma+\frac{\alpha^2b^2}{6}\right)x_4\\
&\quad-\alpha ab^2x_5+(2b^2\beta-\alpha ab^2)x_6+ab^3x_7.
\end{aligned}
\]
Substitution in the rewritten expression cancels the terms involving $\alpha,\beta,\gamma$ and leaves $ab^6P(x)$.

For $x\in A_8^2$, the first two coordinates vanish. The identities for the third and fourth coordinates
follow from the displayed columns of $\varphi$. Substituting the last five coordinates of~\eqref{eq:aut-eight-family} in $K$ and
collecting terms gives
\[
K(\varphi x)=a^2b^7(x_6^3+6x_5x_6^2+9x_5^2x_6-3x_4x_6x_7-9x_4x_5x_7+3x_4^2x_8).
\]
These substitutions establish all the asserted polynomial identities for arbitrary complex parameters
with $ab\ne0$.
\end{proof}

\subsection{Rigidity of local automorphisms}
The next theorem uses only the explicitly computed
automorphisms and their polynomial identities. The details below exhibit the equations
forcing a local automorphism to preserve all products.
\begin{theorem}\label{thm:aut-rigid-eight}
Every local automorphism of $A_8$ is an automorphism. Thus
\[
\LocAut(A_8)=\Aut(A_8).
\]
\end{theorem}
\begin{proof}
Let $T$ be a local automorphism. Choose an automorphism $\varphi_0$ with $T(e_1+e_2)=
\varphi_0(e_1+e_2)$. Replacing $T$ by $\varphi_0^{-1}T$ preserves locality, because at each $x$ a witness $\varphi_x$ is replaced by
$\varphi_0^{-1}\varphi_x$. It is therefore enough to treat the normalization
\[
T(e_1+e_2)=e_1+e_2.
\]
\emph{The generator columns.} Locality at $e_1$ gives $T(e_1)=ae_1+a\alpha e_3+\mu e_8$, so linearity determines
$T(e_2)=e_1+e_2-T(e_1)$. Every automorphism sends $e_1-3e_2$ to a vector with first two coordinates
$(b,-3b)$. On the other hand those coordinates of $T(e_1-3e_2)$ are $(4a-3,-3)$. Comparing gives
$b=1$ for this witness and then $a=1$. Hence, for some $c,d\in\C$,
\[
T(e_1)=e_1+ce_3+de_8,\qquad T(e_2)=e_2-ce_3-de_8.
\]
\emph{The remaining column shapes.} Applying locality separately at $e_3,\ldots,e_8$ and using Proposition~\ref{prop:aut-eight}, write
\[
\begin{aligned}
T(e_3)&=fe_3,\\
T(e_4)&=te_4+ue_5+ve_6+we_7+z_4e_8,\\
T(e_5)&=me_5+ne_7+z_5e_8,\\
T(e_6)&=re_5+se_6+je_7+z_6e_8,\\
T(e_7)&=le_7+z_7e_8,\qquad T(e_8)=s_8e_8.
\end{aligned}
\]
In particular, $T$ fixes the first two coordinates of every vector. If $x_2\ne0$ and $3x_1+x_2\ne0$, a witness
$T(x)=\varphi_x(x)$ must have $a=b=1$, by the first two identities of~\eqref{eq:aut-PQ}. Consequently
\begin{equation}\label{eq:aut-local-PQ}
Q(Tx)=Q(x),\qquad P(Tx)=P(x)
\end{equation}
for every such $x$. All test vectors used in the following paragraph satisfy these two nonvanishing conditions.

\emph{Removing the noncentral errors.} The $Q$ identity at $e_2+e_6$ gives $s=1$. At $e_2+e_4$ and $e_2-e_4$ it
gives respectively
\[
2v-t^2=-1,\qquad-2v-t^2=-1.
\]
Thus $v=0$ and $t^2=1$. The differences $P(Tx)-P(x)$ at $e_2+e_6$ and $2e_1+e_2+e_6$ are
\[
-6c-3j=0,\qquad6c-3j=0,
\]
so $c=j=0$. With these simplifications, the same differences at $e_2+e_4$, $e_2-e_4$, and $e_2+2e_4$ give
\[
\begin{aligned}
1-t^3+3tu-3w&=0,\\
-1+t^3+3tu+3w&=0,\\
8-8t^3+12tu-6w&=0.
\end{aligned}
\]
Adding the first two yields $6tu=0$; since $t^2=1$, we have $u=0$. The first and third equations then
give $w=0$ and $t^3=1$. Together with $t^2=1$ this forces $t=1$.

The final five evaluations of $P(Tx)-P(x)$ are
\[
\begin{array}{c|c}
x&P(Tx)-P(x)\\\hline
e_2+e_3+e_4&3(1-f)\\
e_2+e_5&-3n\\
e_2+e_4+e_5&3(m-n-1)\\
e_2+e_4+e_6&3r\\
e_2+e_7&3(1-l).
\end{array}
\]
They all vanish. Hence $f=m=l=1$ and $n=r=0$. We have proved
\[
T(e_3)=e_3,\qquad T(e_i)=e_i+z_ie_8\ (4\le i\le7),\qquad T(e_8)=s_8e_8.
\]
\emph{Removing the central errors.} For $x\in A_8^2$ with $x_3x_4\ne0$, comparison of coordinates three and
four in $T(x)=\varphi_x(x)$ gives $ab=1$ and $b^2=1$. Thus $(a^2b^7)^2=1$, and~\eqref{eq:aut-K} implies
\begin{equation}\label{eq:aut-K-local}
K(Tx)^2=K(x)^2.
\end{equation}
For all $x\in A_8^2$ and $h\in\C$,
\[
K(x+he_8)=K(x)+3x_4^2h.
\]
At $x=e_3+e_4$, the value $K(x)$ is zero and~\eqref{eq:aut-K-local} gives $(3z_4)^2=0$, so $z_4=0$. Applying it at
$e_3+e_4+e_5$ and $e_3+e_4+e_7$, which also have $K(x)=0$, gives successively $z_5=z_7=0$. At
$e_3+e_4+e_6$ and $e_3+e_4+2e_6$ the same identity gives
\[
(1+3z_6)^2=1,\qquad(8+6z_6)^2=64.
\]
The first equation gives $z_6\in\{0,-\tfrac23\}$ and the second gives $z_6\in\{0,-\tfrac83\}$. Thus $z_6=0$. Finally, the
vectors $e_3+e_4+e_8$ and $e_3+e_4+e_6+e_8$ give
\[
9s_8^2=9,\qquad(1+3s_8)^2=16.
\]
The first allows $s_8=\pm1$, and substitution in the second leaves only $s_8=1$.

We conclude that the normalized map fixes $e_3,\ldots,e_8$ and sends
\[
e_1\longmapsto e_1+de_8,\qquad e_2\longmapsto e_2-de_8.
\]
It is precisely~\eqref{eq:aut-eight-family} with $(a,b,\alpha,\beta,\gamma,\mu,\nu)=(1,1,0,0,0,d,-d)$, so it is an automorphism. Undoing
the normalization proves that the original $T$ is an automorphism. The reverse inclusion is immediate
from the definition of locality.
\end{proof}

\subsection{Minimal dimension and minimal nilpotency class}
We can now combine the constructive
lower bounds and the explicit eight-dimensional example. The two minimality statements concern
all non-abelian nilpotent complex right Leibniz algebras.
\begin{theorem}\label{thm:aut-minima}
The least dimension of such an algebra $A$ satisfying $\LocAut(A)=\Aut(A)$ is eight.
The least possible nilpotency class of such an algebra is five. Both minima are attained by $A_8$.
\end{theorem}
\begin{proof}
Theorem~\ref{thm:aut-small} constructs a proper local automorphism for every algebra of dimension at most
seven. Hence rigidity cannot occur in those dimensions. The algebra $A_8$ has dimension eight, and
Theorem~\ref{thm:aut-rigid-eight} proves its rigidity. This proves the dimension assertion.

Theorem~\ref{thm:aut-small} also excludes rigidity in classes at most four, independently of dimension. Proposition~\ref{prop:min-A8}
computes the class of $A_8$ to be five. Thus the smallest class is five, attained by the same
algebra.
\end{proof}
\begin{corollary}
The single eight-dimensional non-Lie algebra $A_8$ satisfies simultaneously
\[
\LocDer(A_8)=\Der(A_8),\qquad\LocAut(A_8)=\Aut(A_8).
\]
For each equality, eight is the smallest possible dimension and five is the smallest possible nilpotency
class among non-abelian nilpotent complex right Leibniz algebras.
\end{corollary}
\begin{proof}
Combine Theorems~\ref{thm:min-eight} and~\ref{thm:aut-minima}.
\end{proof}
\begin{remark}
The results concern local maps. They neither classify all rigid right Leibniz
algebras in larger dimensions nor make an assertion about potentially nonlinear $2$-local maps. The
two rigidity equalities for $A_8$ have been proved separately; no implication between them for arbitrary
algebras is used.
\end{remark}

\end{document}